\documentclass[10pt]{amsart}

\usepackage[a4paper,margin=2.5cm]{geometry}
\usepackage{amsmath,amssymb,amsthm,mathtools,mathrsfs}
\usepackage[T1]{fontenc}
\usepackage[utf8]{inputenc}
\usepackage{lmodern}
\usepackage{microtype}
\usepackage{hyperref}
\usepackage{enumitem}

\hypersetup{
  colorlinks=false,
  linkbordercolor={1 0 0},
  citebordercolor={0 1 0},
  urlbordercolor={0 1 1}
}

\newtheoremstyle{spacedplain}%
  {3pt plus 1pt minus 1pt}%
  {3pt plus 1pt minus 1pt}%
  {\itshape}%
  {}%
  {\bfseries}%
  {.}%
  {.5em}%
  {}

\newtheoremstyle{spaceddefinition}%
  {\medskipamount}%
  {\medskipamount}%
  {}%
  {}%
  {\bfseries}%
  {.}%
  {.5em}%
  {}

\theoremstyle{spacedplain}
\newtheorem{theorem}{Theorem}[section]
\newtheorem{proposition}[theorem]{Proposition}
\newtheorem{lemma}[theorem]{Lemma}
\newtheorem{corollary}[theorem]{Corollary}

\theoremstyle{spaceddefinition}

\newcommand{\N}{\mathbb{N}}
\newcommand{\dd}{\,\mathrm{d}}
\newcommand{\z}{\zeta}
\newcommand{\1}{\mathbf{1}}
\newcommand{\Om}{\Omega}
\newcommand{\om}{\omega}
\newcommand{\Pplus}{P^{+}}
\newcommand{\eps}{\varepsilon}

\allowdisplaybreaks[2]

\title{Chebyshev Bias for Largest Prime Factors}

\author{Nilotpal Kanti Sinha}
\address{Dubai, UAE}
\email{nilotpal.sinha@gmail.com}
\date{}


\begin{document}

\begin{abstract}
Let $P^+(n)$ be the largest prime factor of $n$, and let $\chi=\chi_{-4}$ be
$1$ on primes $1\pmod4$ and $-1$ on primes $3\pmod4$.  For fixed $k\ge2$ we
study
\[
 D_k(x)=\sum_{\substack{n\le x\\ \Omega(n)=k}}\chi(P^+(n)).
\]
Thus $D_k(x)$ compares the two residue classes according to the largest prime
factor of integers having exactly $k$ prime factors, counted with
multiplicity.  Assuming RH for $\zeta(s)$ and $L(s,\chi)=\beta(s)$, we prove a
pointwise explicit formula.  The main term is a fixed negative contribution
plus an absolutely convergent oscillating sum over the zeros
$\rho=\frac12+i\gamma$ of $L(s,\chi)$.  The ordering of the prime factors
produces $k$ Perron denominators, and the principal coefficient of a zero is
$O_k((1+|\gamma|)^{-k})$.  We then show that the total size of all zero terms
is strictly smaller than the fixed contribution.  Hence $D_k(x)<0$ for all
sufficiently large $x$.  The analogous problem without fixing $k$ is still
open.
\end{abstract}

\maketitle

\section{Introduction}\label{sec:intro}

For an integer $n\ge2$, let $\Pplus(n)$ be its largest prime factor.  We write
$\Om(n)$ for the number of prime factors counted with multiplicity and
$\om(n)$ for the number of distinct prime factors.  For example,
$\Om(12)=3$ but $\om(12)=2$ because $12=2^2\cdot3$.  Throughout,
\[
 \chi=\chi_{-4},\qquad
 \chi(n)=
 \begin{cases}
  0,&2\mid n,\\
  1,&n\equiv1\pmod4,\\
 -1,&n\equiv3\pmod4.
 \end{cases}
\]
For fixed $k\ge1$ we study
\begin{equation}\label{eq:Dk-def}
 D_k(x):=\sum_{\substack{n\le x\\ \Om(n)=k}}
 \chi\bigl(\Pplus(n)\bigr).
\end{equation}
Thus $D_k(x)$ is simply a signed count.  Each integer contributes $+1$ when
its largest odd prime factor is $1\pmod4$, contributes $-1$ when it is
$3\pmod4$, and contributes $0$ when the largest prime factor is $2$.
Therefore $D_k(x)<0$ means that the class $3\pmod4$ occurs more often in this
fixed-$k$ layer.

The corresponding sum over all integers is
\[
 D(x):=\sum_{2\le n\le x}\chi\bigl(\Pplus(n)\bigr).
\]
Earlier work shows that, if one only looks at the leading term, the two odd
classes $1$ and $3$ modulo $4$ occur equally often as largest prime factors;
see \cite{Oon04,BanksHarmanShparlinski05,DeKoninckKatai11,PollackSinghaRoy23}.
What remains is the much smaller difference between the two counts.  Numerically, however, the race shows a striking persistent bias: for
$n\le 3\times10^{10}$, there are $14{,}970{,}975{,}209$ integers whose
largest prime factor is $1\pmod 4$ and $15{,}029{,}024{,}755$ whose
largest prime factor is $3\pmod 4$, with no tie or lead by the
$1\pmod 4$ class observed in this range \cite{SinhaMO22}. Even the
sign of the unrestricted difference $D(x)$ is unknown; it appears as
Problem~18 in the Comparative Prime Number Theory problem list
\cite{SinhaMO22,HamiehEtAl24}.  We therefore fix $k$.  In this thinner set of
integers, the secondary term can be described pointwise.

The statistic here is not the usual prime-race statistic.  In the work of
Ford--Sneed \cite{FordSneed10} and Meng \cite{Meng18}, one studies the residue
class of the whole product of primes.  Here only the \emph{largest} prime
factor is tested.  The distinction is important because the primes can no
longer be handled symmetrically: the largest one must remain distinguished.

The form of the conclusion also differs.  Ford and Sneed obtain a limiting
distribution for a semiprime race under an extended
Riemann hypothesis and a linear-independence assumption, while Meng uses a
one-variable generating series and Hankel contours for fixed almost-prime
layers.  Our result gives a formula for each sufficiently large $x$, not only
an averaged or limiting statement.  Its complete zero contribution converges
absolutely, so no linear-independence assumption on zero ordinates is needed.
To the best of our knowledge, no earlier pointwise explicit formula treats a
fixed-$\Om$ layer with a character applied specifically to its largest prime
factor.

The ordering is the main analytic complication, but it also supplies the
decay needed in the proof.  For the moment suppose the prime factors are
distinct and write them as $p_1<\cdots<p_k$.  Perron's formula is a standard
way of turning an inequality into a complex contour integral.  We use one
Perron variable for the product condition $p_1\cdots p_k\le x$ and another
$k-1$ variables to enforce $p_1<\cdots<p_k$.  These $k$ variables produce
$k$ factors in the denominator of the contour kernel.  When a zero
$\rho=\frac12+i\gamma$ of $L(s,\chi)$ is encountered, those denominators make
its coefficient decay like $O_k(|\gamma|^{-k})$.  Because the number of zeros
up to height $T$ is only $O(T\log T)$, the complete zero series is absolutely
convergent for every $k\ge2$.  This denominator effect is the central
structural idea of the paper.

From now on, RH means that the nontrivial zeros of the function in question
lie on the line $\Re s=1/2$.  We use standard asymptotic notation:
$A\ll_k B$ or $A=O_k(B)$ means $|A|\le C_kB$ for a constant depending only
on $k$.  Write $\rho=\frac12+i\gamma$ for the distinct
nontrivial zeros of $L(s,\chi)=\beta(s)$, and let $m_\rho$ be their
multiplicities.  For $k\ge2$, put
\begin{equation}\label{eq:B-C-def}
 B_k:=\frac{2^{k-1}k^{2k-1}}{(k-1)!(2k-1)},\qquad
 C_k(z):=\frac{k^{2k-1}}
 {(k-1)!(k-1+z)(1-z)^{k-1}}.
\end{equation}

\begin{theorem}\label{thm:main-intro}
Assume RH for both $\zeta(s)$ and $L(s,\chi)=\beta(s)$.  For every fixed $k\ge2$,
\begin{equation}\label{eq:main-intro}
 D_k(x)=-
 \frac{x^{1-\frac1{2k}}}{(\log x)^k}
 \left\{
 B_k+2\Re\sum_{\gamma>0}
 m_\rho C_k(\rho)x^{i\gamma/k}
 \right\}
 +O_k\left(\frac{x^{1-\frac1{2k}}}{(\log x)^{k+1}}\right).
\end{equation}
The zero series converges absolutely and uniformly as a function of $\log x$.
\end{theorem}

The constant $B_k$ comes from the prime-square part of the Euler product and
gives a fixed negative bias after
the outside minus sign is applied.  The zeros of $L(s,\chi)$ produce the
oscillating part.  We will bound all of those oscillations at once and prove
that, even if their phases align in the worst possible way, they cannot
cancel the fixed negative term.  This gives the following sign result.

\begin{corollary}\label{cor:negative-intro}
Under the assumptions of Theorem~\ref{thm:main-intro}, for every fixed
$k\ge2$ there is $x_k$ such that
\[
 D_k(x)<0\qquad(x\ge x_k).
\]
\end{corollary}

The proof proceeds in three steps.  First, we replace the product and
ordering conditions by a multidimensional contour integral.  Second, we move the
contours to the left.  Whenever a contour crosses a singularity, its residue
becomes a term in the formula.  The singularities at $1$, at $1/2$, and at
the zeros of $L(s,\chi)$ give the terms of leading size; every other piece is
smaller by a fixed power of $x$.  Third, we remove the temporary logarithmic
weights and prove that the total zero contribution is smaller than the fixed
term.  All estimates are for fixed $k$.  They are not uniform when
$k$ is of order $\log\log x$, which is why the unrestricted sign problem remains
open.

\section{Prime series and the ordered Perron formula}
\label{sec:perron}

We first put the counting problem into a form suited to complex analysis by
temporarily attaching a factor $\log p$ to each prime.
This makes the prime sums become standard logarithmic derivatives of Euler
products.  Perron's formula then converts the product and ordering
inequalities into contour integrals.  The temporary logarithmic weights are
removed at the end.

\subsection{Prime logarithmic series}

For $\Re s>1$, define
\[
 F_0(s):=\sum_p\frac{\log p}{p^s},\qquad
 F_\chi(s):=\sum_p\frac{\chi(p)\log p}{p^s}.
\]

\begin{lemma}[Prime series and their poles]
\label{lem:prime-series}
For every fixed $\delta>0$, there are functions $H_0,H_\chi$, holomorphic
and bounded on every closed half-plane
$\Re s\geq1/3+\delta$, such that
\begin{align}
 F_0(s)&=-\frac{\z'}{\z}(s)+\frac{\z'}{\z}(2s)+H_0(s),
 \label{eq:F0}\\
 F_\chi(s)&=-\frac{L'}{L}(s,\chi)
 +\frac{\z'}{\z}(2s)+H_\chi(s).
 \label{eq:Fchi}
\end{align}
In this half-plane the complete pole inventory, before any accidental
cancellation at coincident points, is as follows.
\begin{enumerate}[label=\textup{(\roman*)}]
\item $F_0$ has a simple pole of residue $1$ at $s=1$.
\item If $\varrho$ is a zero of $\z(s)$ of multiplicity $m_\varrho$, then
$-\z'/\z(s)$ contributes residue $-m_\varrho$ to $F_0$ at
$s=\varrho$; and $\z'/\z(2s)$ contributes residue $m_\varrho/2$ to
both $F_0$ and $F_\chi$ at $s=\varrho/2$.
\item The pole of $\z(2s)$ at $s=1/2$ gives both $F_0$ and $F_\chi$ a
simple pole of residue $-1/2$ there.
\item If $\rho$ is a zero of $L(s,\chi)$ of multiplicity $m_\rho$, then
$-L'/L(s,\chi)$ contributes residue $-m_\rho$ to $F_\chi$ at
$s=\rho$.
\end{enumerate}
There are no other poles in $\Re s>1/3+\delta$.
\end{lemma}

\begin{proof}
The logarithmic derivatives of the Euler products give
\[
 -\frac{\z'}{\z}(s)
 =\sum_p\sum_{\nu\geq1}\frac{\log p}{p^{\nu s}},
\qquad
 -\frac{L'}{L}(s,\chi)
 =\sum_{p\nmid4}\sum_{\nu\geq1}
 \frac{\chi(p)^\nu\log p}{p^{\nu s}}.
\]
Separate the terms $\nu=1,2$.  Since $\chi(p)^2=1$ for odd $p$, the
$\nu=2$ terms in both formulas have the same singular part as
$-\z'(2s)/\z(2s)$.  All terms with $\nu\geq3$, together with the prime-$2$
correction, converge absolutely and uniformly on compact subsets of every
closed half-plane
$\Re s\geq1/3+\delta$.  This proves \eqref{eq:F0} and \eqref{eq:Fchi} and
the asserted boundedness of the two remainders.

At a zero of multiplicity $m$, a logarithmic derivative has residue $m$;
at the pole of $\z$ at $1$, it has residue $-1$.  The change of variable
$s\mapsto2s$ divides residues by $2$.  These observations give the
listed singular parts and show that the list is complete.
\end{proof}

\subsection{The ordered integral}

We first treat integers whose $k$ prime factors are distinct.  Such integers
are called squarefree in this fixed layer.  The parameters
$u_1,\ldots,u_k$ below are only bookkeeping variables: integrating them out
will exactly remove the logarithmic weights.  For
$\boldsymbol u=(u_1,\ldots,u_k)\in[0,\infty)^k$, put
\begin{equation}\label{eq:Gk}
 G_k(X;\boldsymbol u):=
 \sum_{\substack{p_1<\cdots<p_k\\p_1\cdots p_k\leq X}}
 \chi(p_k)\prod_{j=1}^k\frac{\log p_j}{p_j^{u_j}}.
\end{equation}
Termwise integration gives
\begin{equation}\label{eq:remove-weights}
 \int_{[0,\infty)^k}G_k(X;\boldsymbol u)\dd\boldsymbol u
 =
 D_k^\square(X),
\end{equation}
where
\[
 D_k^\square(X):=
 \sum_{\substack{n\leq X\\\Om(n)=\om(n)=k}}
 \chi(\Pplus(n)).
\]
Here the superscript $\square$ indicates that the sum is restricted to squarefree integers.

Let
\[
 h_T(y):=\frac{1}{2\pi i}
 \int_{c-i\infty}^{c+i\infty}
 y^s e^{s^2/T^2}\frac{\dd s}{s}\qquad(c>0).
\]
Direct differentiation, followed by evaluation at $y=1$, gives
\begin{equation}\label{eq:hT}
 h_T(y)=\frac12\operatorname{erfc}
 \left(-\frac{T\log y}{2}\right).
\end{equation}
Here $\operatorname{erfc}$ is the complementary error function.
Thus $h_T$ is a smooth version of the step function that tests whether
$y>1$.  When $T$ is large it is very close to $1$ for $y>1$ and very close
to $0$ for $y<1$, with a rapidly decaying Gaussian error away from the
boundary $y=1$.  This smoothing lets us use contour integrals without having
to carry sharp cutoffs through every intermediate step.

Introduce variables $s,w_1,\ldots,w_{k-1}$ for the product condition and
the $k-1$ order conditions.  Set
\begin{equation}\label{eq:a-coordinates}
 a_1=s+w_1,\qquad
 a_j=s-w_{j-1}+w_j\ (2\leq j\leq k-1),\qquad
 a_k=s-w_{k-1}.
\end{equation}
The inverse transformation is
\begin{equation}\label{eq:inverse-coordinates}
 s=\frac1k\sum_{j=1}^ka_j,\qquad
 w_j=\sum_{i\leq j}a_i-js\quad(1\leq j\leq k-1),
\end{equation}
and its Jacobian is
\begin{equation}\label{eq:jacobian}
 \dd s\,\dd w_1\cdots\dd w_{k-1}
 =\frac1k\,\dd a_1\cdots\dd a_k.
\end{equation}
Define
\begin{equation}\label{eq:kernel}
 K_k(\boldsymbol a):=
 \frac{1}{k\,s(\boldsymbol a)
 \prod_{j=1}^{k-1}w_j(\boldsymbol a)}
\end{equation}
and
\[
 Q_k(\boldsymbol a):=
 s(\boldsymbol a)^2+\sum_{j=1}^{k-1}w_j(\boldsymbol a)^2.
\]

Replacing each sharp inequality by this smooth test function gives a sum
that is easier to represent by contour integrals.  The resulting smoothing
error will be shown to be negligible.
Define
\begin{equation}\label{eq:regularized-G}
 \mathcal G_{k,T}(X;\boldsymbol u):=
 \sum_{p_1,\ldots,p_k\ \mathrm{prime}}
 \chi(p_k)\prod_{j=1}^k\frac{\log p_j}{p_j^{u_j}}
 h_T\left(\frac{X}{p_1\cdots p_k}\right)
 \prod_{j=1}^{k-1}h_T\left(\frac{p_{j+1}}{p_j}\right).
\end{equation}
Here the primes are summed independently; the ordering is imposed only by the
factors $h_T(p_{j+1}/p_j)$.  For finite $T$ these smooth factors do not
completely exclude equal neighboring primes.  Those equalities are what we
later call the \emph{diagonal terms}.  They will be isolated and shown to be
smaller than the main term.

\begin{proposition}[Regularized ordered Perron formula]
\label{prop:perron}
Fix $\sigma_0>\sigma_1>1$, and take the initial $a$-contours
$\Re a_1=\cdots=\Re a_{k-1}=\sigma_0$ and
$\Re a_k=\sigma_1$.  Then
\begin{equation}\label{eq:perron}
\begin{aligned}
 \mathcal G_{k,T}(X;\boldsymbol u)
 ={}&\frac{1}{(2\pi i)^k}\int\cdots\int
 F_0(a_1+u_1)\cdots F_0(a_{k-1}+u_{k-1})\\
 &\times F_\chi(a_k+u_k)X^{s(\boldsymbol a)}
 e^{Q_k(\boldsymbol a)/T^2}K_k(\boldsymbol a)\dd\boldsymbol a.
\end{aligned}
\end{equation}
\end{proposition}

\begin{proof}
Use one copy of the integral defining $h_T$ for the product condition and one
copy for each of the $k-1$ order conditions.  On the initial contours all
prime series converge absolutely, so we may interchange the prime sums and
the contour integrals.  The exponent of $p_1$ is $-(s+w_1)$, the exponent of
$p_j$ is $-(s-w_{j-1}+w_j)$ for $2\leq j\leq k-1$, and the exponent of
$p_k$ is $-(s-w_{k-1})$.  Hence the prime sums factor into the $F$-functions
in \eqref{eq:perron}.  The change of variables
\eqref{eq:a-coordinates}--\eqref{eq:jacobian} gives the kernel $K_k$ and the
Gaussian factor $Q_k$.
\end{proof}

\begin{lemma}[Uniform desmoothing]\label{lem:desmoothing}
Let $X=N+\tfrac12$, $N\in\N$, and put
\[
 T=\exp\{(\log X)^2\}.
\]
Define the exact diagonal contribution
\begin{equation}\label{eq:diagonal-definition}
 \mathcal E_k^\Delta(X;T):=
 \sum_{\substack{p_1,\ldots,p_k\\ \text{prime}\\
                  p_j=p_{j+1}\\ \text{for some}\ j<k}}
 \chi(p_k)
 h_T\left(\frac{X}{p_1\cdots p_k}\right)
 \prod_{j<k}h_T\left(\frac{p_{j+1}}{p_j}\right).
\end{equation}
After integration over $\boldsymbol u\in[0,\infty)^k$, the smoothed sum in
Proposition~\ref{prop:perron} returns the original squarefree count, apart
from the diagonal terms just described and an error smaller than every fixed
negative power of $X$:
\begin{equation}\label{eq:desmoothing}
 D_k^\square(X)+\mathcal E_k^\Delta(X;T)+O_{k,M}(X^{-M})
\end{equation}
for every fixed $M>0$.  The same conclusion holds after multiplication of
each prime tuple by any weight $W(\boldsymbol p)$ satisfying
$|W(\boldsymbol p)|\leq1$: the first term becomes the corresponding
weighted sharp sum, and the second becomes the exact adjacent-equality sum
in \eqref{eq:diagonal-definition} with the same weight inserted.
\end{lemma}

\begin{proof}
Termwise integration in $\boldsymbol u$ is justified by positivity after
absolute values and gives
\[
 \int_0^\infty(\log p)p^{-u}\dd u=1.
\]
Thus the integrated regularized sum is
\[
 \sum_{p_1,\ldots,p_k}\chi(p_k)
 h_T\left(\frac{X}{p_1\cdots p_k}\right)
 \prod_{j<k}h_T\left(\frac{p_{j+1}}{p_j}\right).
\]
Here $\mathbf{1}_E$ denotes the indicator of a set $E$.
For $v\neq1$, the standard complementary-error-function estimate gives
\begin{equation}\label{eq:h-error}
 \left|h_T(v)-\1_{(1,\infty)}(v)\right|
 \ll
 \frac{\exp\{-T^2(\log v)^2/4\}}
 {1+T|\log v|}.
\end{equation}

Put $P=p_1\cdots p_k$.  Suppose first that $P\leq2X$ and that no adjacent
coordinates coincide.  Since $P$ is an integer and
$X=N+\tfrac12$,
\[
 |\log(X/P)|\gg X^{-1}.
\]
Also every $p_j\leq2X$, and for distinct primes $p,q\leq2X$,
\[
 |\log(q/p)|\gg X^{-1}.
\]
The elementary telescoping inequality
\[
 \left|\prod_{j=0}^{k-1}a_j-\prod_{j=0}^{k-1}b_j\right|
 \leq\sum_{j=0}^{k-1}|a_j-b_j|,
 \qquad 0\leq a_j,b_j\leq1,
\]
combined with \eqref{eq:h-error}, shows that the total discrepancy from
such tuples is
\[
 \ll_k X^k\exp\{-c_kT^2/X^2\}=O_{k,M}(X^{-M}).
\]

If $P>2X$, the sharp product indicator vanishes and the absolute value of
the regularized weight is at most $h_T(X/P)$.  Grouping according to
$2^rX<P\leq2^{r+1}X$ and majorizing the number of ordered prime
factorizations of an integer $m$ by $d_k(m)\leq m^k$, we obtain
\[
 \sum_{P>2X}h_T(X/P)
 \ll_k
 \sum_{r\geq1}(2^{r+1}X)^{k+1}
 \exp\{-cT^2r^2\}
 =O_{k,M}(X^{-M}).
\]
The remaining tuples are precisely those with an adjacent equality, and
their contribution is exactly the sum in
\eqref{eq:diagonal-definition}.  Inserting a weight of modulus at most one
does not alter any of the absolute estimates above, which proves the
weighted assertion.
\end{proof}

\begin{lemma}[Value of the kernel]\label{lem:kernel-value}
For $z\neq1,1-k$,
\begin{equation}\label{eq:kernel-value}
 K_k(1,\ldots,1,z)=
 \frac{k^{k-1}}
 {(k-1)!(k-1+z)(1-z)^{k-1}}.
\end{equation}
\end{lemma}

\begin{proof}
At $\boldsymbol a=(1,\ldots,1,z)$,
\[
 s=\frac{k-1+z}{k},\qquad
 w_j=\frac{j(1-z)}{k}\quad(1\leq j\leq k-1).
\]
Substitution into \eqref{eq:kernel} gives \eqref{eq:kernel-value}.
\end{proof}

\section{Moving the contours and bounding the errors}\label{sec:contours}

We extract the main terms from the contour integral by Cauchy's residue
theorem.  When a contour crosses a pole, that pole contributes its residue to
the final arithmetic formula.  Because there are $k$ contour
variables, some coordinates may already have been replaced by residues while
others are still being integrated.  We call each such mixed configuration a
``residue stratum''.

In this section $k\geq2$ is fixed, and RH is assumed for both $\z$ and
$L(s,\chi)=\beta(s)$.  The two assumptions play different roles: RH for
$L(s,\chi)$ supplies cancellation in the signed prime coordinate, while RH
for $\zeta(s)$ lets the ordinary prime contours move a fixed distance to the
left of $1$ without crossing zeta zeros.  Set
\begin{equation}\label{eq:eps-eta}
 \eps_k:=\frac{1}{16k},\qquad
 \eta_k:=\frac{1}{64k^2}.
\end{equation}
The leading contribution comes from small values of the auxiliary variables
$u_j$.  We therefore first integrate \eqref{eq:perron} over the small cube
$0\leq u_j\leq\eta_k$.  The remaining region, where at least one $u_j$ is
larger, will be shown later to have a fixed power saving.

The last prime $p_k$ is special because it carries the character $\chi$.
We therefore call its contour the \emph{signed contour}.  For fixed
$\boldsymbol u$ in the cube, move this contour first to
\[
 \Re(a_k+u_k)=\frac12-\eps_k.
\]
By Lemma~\ref{lem:prime-series}, this crosses the pole at $1/2$, of
residue $-1/2$, and every nontrivial zero
$\rho=1/2+i\gamma$ of $L(s,\chi)$, of residue $-m_\rho$.  It crosses no
half-zero pole from $\z(2s)$, since RH puts those poles on
$\Re s=1/4<1/2-\eps_k$.  Next move the
ordinary contours, in the reverse order
$a_{k-1},a_{k-2},\ldots,a_1$, to
\[
 \Re(a_j+u_j)=1-\eps_k.
\]
Again by Lemma~\ref{lem:prime-series}, only the pole at $1$ is crossed:
all other poles of $F_0$ lie on or to the left of $\Re s=1/2$, whereas
$1-\eps_k>1/2$.
There is no central zero of $L(s,\chi)$: the alternating-series estimate
\[
 L(1/2,\chi)=\beta(1/2)>
 1-\frac{1}{\sqrt3}>0
\]
shows that no zero residue occurs at $\gamma=0$.

\begin{lemma}[No kernel pole is crossed]\label{lem:geometry}
During the preceding contour displacement,
\[
 \Re s(\boldsymbol a)>0,\qquad
 \Re w_j(\boldsymbol a)>0\quad(1\leq j\leq k-1).
\]
On every final contour and every residue stratum,
\[
 \Re w_j(\boldsymbol a)\geq\frac{j}{4k}.
\]
\end{lemma}

\begin{proof}
The signed contour is moved first.  While all ordinary coordinates remain
on $\Re a_i=\sigma_0$, lowering $\Re a_k$ can only increase every
$\Re w_j$.

Now proceed by reverse induction through the ordinary coordinates.  At the
stage when $a_j$ is moved, the later ordinary coordinates have already
either been placed on their final lines or replaced by residues; the earlier
coordinates remain on their initial lines.  If $\ell<j$, all coordinates
in the prefix defining $w_\ell$ are still on the line $\sigma_0$, while
every suffix coordinate has real part at most $\sigma_0$ and at least one
has smaller real part.  Hence $\Re w_\ell>0$; lowering $a_j$ only increases
it.  If $\ell\geq j$, increasing a prefix coordinate increases $w_\ell$,
so its smallest possible real part is bounded by the configuration in which
all prefix coordinates are as small as allowed, all ordinary suffix
coordinates are as large as allowed, and the signed coordinate has real
part $1/2$.  With
$\delta_k=\eps_k+\eta_k$, \eqref{eq:inverse-coordinates} gives
\[
 k\Re w_\ell
 \geq \ell\{(k-\ell)(1-\delta_k)
 -(k-1-\ell)-1/2\}
 =\ell\{1/2-(k-\ell)\delta_k\}.
\]
Since $k\delta_k<1/8$, this is at least $\ell/4$.  The same argument, or
the positivity of the average of the coordinates, gives $\Re s>0$.
On a final contour or residue stratum every ordinary prefix coordinate is
at least $1-\delta_k$, every ordinary suffix coordinate is at most $1$,
and the signed coordinate is at most $1/2$; the displayed calculation
therefore also proves the asserted uniform lower bound.
\end{proof}

\begin{lemma}[Decay of the principal zero coefficient]\label{lem:principal-zero-decay}
Let $\rho=1/2+i\gamma$ be a zero of $L(s,\chi)$ and put
\[
 \boldsymbol a_\rho(\boldsymbol u)
 =(1-u_1,\ldots,1-u_{k-1},\rho-u_k),
 \qquad 0\le u_j\le\eta_k.
\]
Then, uniformly on this cube,
\begin{equation}\label{eq:principal-zero-decay}
 |K_k(\boldsymbol a_\rho(\boldsymbol u))|
 +\|\nabla_{\boldsymbol u}K_k(\boldsymbol a_\rho(\boldsymbol u))\|
 \ll_k (1+|\gamma|)^{-k}.
\end{equation}
Consequently
\[
 \sum_\rho m_\rho
 \sup_{\boldsymbol u\in[0,\eta_k]^k}
 |K_k(\boldsymbol a_\rho(\boldsymbol u))|<\infty.
\]
\end{lemma}

\begin{proof}
For the displayed vector $\boldsymbol a_\rho(\boldsymbol u)$, the only
imaginary part is $\gamma$ in the last coordinate.  From
\eqref{eq:inverse-coordinates},
\[
 \Im s=\frac{\gamma}{k},\qquad
 \Im w_j=-\frac{j\gamma}{k}\quad(1\le j<k).
\]
Lemma~\ref{lem:geometry} gives positive lower bounds for the real parts of
$s,w_1,\ldots,w_{k-1}$.  Hence every one of the $k$ factors in the
denominator of $K_k$ has size $\gg_k1+|\gamma|$, and therefore
$|K_k|\ll_k(1+|\gamma|)^{-k}$.  Differentiating with respect to a $u_j$
only differentiates linear denominator factors, so the same bound, in fact a
slightly stronger one for large $|\gamma|$, holds for the gradient.  Finally,
$N_\chi(T+1)-N_\chi(T)\ll\log(2+T)$, and
\[
 \sum_{n\ge0}\frac{\log(2+n)}{(1+n)^k}<\infty
 \qquad(k\ge2).
\]
This proves the last assertion.
\end{proof}

The preceding lemma concerns the principal zero stratum, where all ordinary
coordinates have already been replaced by their residues.  A partial zero
stratum may still contain free contour variables, so coefficientwise
$|\gamma|^{-k}$ decay is neither needed nor used there.  The next proposition
handles all such mixed strata at once by integrating the free variables and
summing the zero ordinates under the Gaussian regularization.

\begin{proposition}[Uniform contour and zero-stratum bound]\label{prop:stratum-majorant}
Fix $\boldsymbol u\in[0,\eta_k]^k$ and consider any final-contour or residue
term produced by the shifts below.  Let $\beta$ be the real part of
$s(\boldsymbol a)$ on that term.  The factor $X^\beta$ gives the natural size
of the stratum before logarithmic factors.  Its total absolute contribution is
\begin{equation}\label{eq:uniform-stratum-integral}
 \ll_k X^\beta(\log T)^{M_k}
\end{equation}
for some $M_k>0$, uniformly in $\boldsymbol u$.  If the signed coordinate is
a zero residue, the bound includes the sum over all zeros, with multiplicity.
\end{proposition}

\begin{proof}
Let $d_c$ be the number of coordinates that still run on vertical contours.
If the signed coordinate is not a zero residue, put $d=d_c$.  If it is a zero
residue and $d_c\ge1$, put $d=d_c+1$ and count the zero ordinate $\gamma$ as
one additional vertical variable.  The principal zero stratum $d_c=0$ is
handled separately by Lemma~\ref{lem:principal-zero-decay}.

We need $d$ genuinely independent denominator factors.  The map
\[
 (a_1,\ldots,a_k)\longmapsto(s,w_1,\ldots,w_{k-1})
\]
is invertible by \eqref{eq:a-coordinates}--\eqref{eq:jacobian}.  Therefore
the columns belonging to the free contour coordinates, together with the
signed column when $\gamma$ is counted, are linearly independent.  Among the
$k$ kernel denominators we can consequently choose $d$ whose imaginary parts
are independent real linear forms $\ell_1,\ldots,\ell_d$ in these vertical
variables.  Lemma~\ref{lem:geometry} gives a positive lower bound, depending
only on $k$, for their real parts and for every unused kernel denominator.

If $d_c=0$ and the signed coordinate is a zero residue, the proposition
follows at once from Lemma~\ref{lem:principal-zero-decay} and the Gaussian.
We may therefore assume that either there is no zero residue or $d_c\ge1$.

On the final lines, standard local partial-fraction formulas and the zero
count $N_F(t+1)-N_F(t)\ll\log(2+|t|)$ give
\[
 F_0(1-\eps_k+it),\quad F_\chi(1/2-\eps_k+it)
 \ll_k\log^2(2+|t|).
\]
The same bound holds for the $\z'(2s)/\z(2s)$ terms.  Also, since the above
linear map is invertible, the Gaussian satisfies
\[
 \left|e^{Q_k(\boldsymbol a)/T^2}\right|
 \ll_k e^{-c_k\|\boldsymbol y\|^2/T^2}
\]
on every stratum, where $\boldsymbol y$ collects the continuous vertical
variables and, for the purpose of the zero-stratum majorant, also the real
variable that replaces $\gamma$.  Thus the absolute integrand is bounded by
\[
 X^\beta\frac{\log^{2k}(2+\|\boldsymbol y\|)
 e^{-c_k\|\boldsymbol y\|^2/T^2}}
 {\prod_{r=1}^d|c_r+i\ell_r(\boldsymbol y)|},
 \qquad c_r\ge c_k>0.
\]
After an invertible real change of variables, the integral is bounded by a
product of one-dimensional integrals of the form
\[
 \int_{\mathbb R}\frac{e^{-ct^2/T^2}}{(c_0^2+t^2)^{1/2}}\,dt
 \ll_c\log T,
\]
with only a fixed extra power of $\log T$ from the numerator.  This proves
\eqref{eq:uniform-stratum-integral} when all $d$ vertical variables are
continuous.

If the signed coordinate has been replaced by a zero
$\rho=1/2+i\gamma$ and some ordinary coordinates are still free, we do not
estimate each zero separately.  Instead the ordinate $\gamma$ is treated as
one more vertical variable.  This gives a direct bound for the sum of the
whole zero family on that stratum.  The principal zero stratum, where no
ordinary variable remains, has the sharper coefficientwise decay proved in
Lemma~\ref{lem:principal-zero-decay}.  Let $G(\boldsymbol
t,\gamma)$ denote the preceding majorant with the zero ordinate in place of
one free variable.  Choose a fixed $\delta>0$ so small that, for
$|v-\gamma|\leq\delta$, all selected denominators and the Gaussian are
comparable after decreasing the Gaussian constant if necessary.  Then
\[
 G(\boldsymbol t,\gamma)
 \ll_k \int_{\gamma-\delta}^{\gamma+\delta}G^\ast(\boldsymbol t,v)\,dv,
\]
where $G^\ast$ has the same shape as $G$.  In each interval $[n,n+1)$ there
are $O(\log(2+|n|))$ zeros, counted with multiplicity.  For the zeros with
$\gamma\in[n,n+1)$, the displayed bound is at most a constant times
$\log(2+|n|)$ times the integral over the enlarged interval
$[n-\delta,n+1+\delta]$.  These intervals, now indexed by $n$, have bounded
overlap.  Summing over $n$ therefore gives the same $d$-dimensional
continuous integral, with one additional logarithmic factor.  Hence the zero strata are also
$O_k(X^\beta(\log T)^{M_k})$, absolutely and uniformly in
$\boldsymbol u$.  This is the absolute convergence used later.
\end{proof}

\begin{lemma}[Safe contour heights]\label{lem:admissible-heights}
There is a sequence $\tau_\nu\to\infty$ such that the horizontal sides of
all contour rectangles avoid the zeros of $\zeta(s)$, $L(s,\chi)$, and the
scaled zeros occurring in $\zeta(2s)$.  The same sequence works for every
coordinate and every $\boldsymbol u\in[0,\eta_k]^k$, and uniformly in the
fixed real strips used below,
\begin{equation}\label{eq:log-derivative-horizontal}
 \frac{\z'}{\z}(\sigma+i\tau_\nu),\quad
 \frac{L'}{L}(\sigma+i\tau_\nu,\chi),\quad
 \frac{\z'}{\z}(2\sigma+2i\tau_\nu)
 \ll_k\log^2(2+\tau_\nu).
\end{equation}
\end{lemma}

\begin{proof}
The lemma chooses horizontal edges that stay away from the zeros, where a
logarithmic derivative could become large.  This is the standard
admissible-height argument.  For a fixed Dirichlet
$L$-function and for $\zeta$, the local zero count in an interval of bounded
length is $O(\log T)$, and the usual partial-fraction formula gives
$F'/F(\sigma+iT)\ll\log^2 T$ whenever $T$ stays a distance
$\gg1/\log T$ from the relevant zero ordinates; see
\cite[Proposition~5.7]{IwaniecKowalski04} and
\cite[Lemmas~12.1--12.2]{MontgomeryVaughan07}.  In each interval
$[R+1/4,R+3/4]$ remove neighborhoods of radius $c/\log R$ around the
ordinates of zeros of $\zeta$, of $L(s,\chi)$, and around half the ordinates
of zeros of $\zeta$.  For small fixed $c>0$ the removed length is $<1/2$, so
a point $\tau_R$ remains.  Choosing one such point for each sufficiently
large integer $R$ and relabeling the resulting increasing sequence as
$\tau_\nu$ gives a single admissible sequence for all three functions and
proves \eqref{eq:log-derivative-horizontal}.
\end{proof}

\begin{lemma}[The contour shifts are valid]\label{lem:shift-justification}
The contour displacements preceding Lemma~\ref{lem:geometry} may be carried
out successively.  Their residue expansions converge absolutely for each
fixed $T$, and the limiting final-contour integrals are independent of the
heights used to form the intermediate rectangles.
\end{lemma}

\begin{proof}
Choose an admissible height $\tau_\nu$ as in
Lemma~\ref{lem:admissible-heights}, and truncate every vertical line there.
The integrand is meromorphic in the resulting finite region.  By
Lemma~\ref{lem:geometry}, the Perron kernel has no pole there.  By
Lemma~\ref{lem:prime-series}, we know every prime-series pole that is
crossed.  We may therefore apply Cauchy's theorem one coordinate at a time.
If a zero lies on a vertical boundary, move that boundary slightly and then
pass to the limit.  The residue is then counted with its multiplicity.

Because $Q_k$ is a positive-definite real quadratic form after restriction
to imaginary coordinates,
\[
 \left|e^{Q_k(\boldsymbol\sigma+i\boldsymbol t)/T^2}\right|
 =e^{Q_k(\boldsymbol\sigma)/T^2-Q_k(\boldsymbol t)/T^2}.
\]
On a horizontal face, this decay is
$\exp\{-c_k\tau_\nu^2/T^2\}$, while
Lemma~\ref{lem:admissible-heights}, the identities
\eqref{eq:F0}--\eqref{eq:Fchi}, and the bounded holomorphic remainders
bound every prime-series factor by a power of
$\log(2+\tau_\nu)$.  The remaining truncated coordinates range
over intervals of length $O(\tau_\nu)$, and every kernel denominator has
real part bounded below as in Lemma~\ref{lem:geometry}.  Thus, for fixed
$X$ and $T$, the integral over any horizontal face is
\[
 \ll_{k,X,T}
 e^{-c_k\tau_\nu^2/T^2}(1+\tau_\nu)^{A_k}
 \longrightarrow0
\]
for some $A_k>0$.  On final faces and residue strata,
Proposition~\ref{prop:stratum-majorant} provides a majorant uniform in
$\boldsymbol u$ and in the truncation height.  This proves absolute
convergence for fixed $T$ and independence of the admissible sequence.
\end{proof}

\begin{proposition}[All other residue terms are smaller]\label{prop:remainders}
Let
\[
 \alpha_k:=1-\frac{1}{2k}.
\]
The leading scale is $X^{\alpha_k}$.  The only strata that can reach
this scale are those in which every ordinary prime factor contributes its
pole at $1$ and the signed factor contributes either the pole at $1/2$ or a
zero of $L(s,\chi)$.  Every other final-contour or partial-residue term is
smaller by a fixed power of $X$:
\begin{equation}\label{eq:remainder-power}
 O_k\left(
 X^{\alpha_k-\kappa_k}(\log X)^{M_k^{\mathrm{rem}}}\right)
\end{equation}
for some $\kappa_k>0$ and $M_k^{\mathrm{rem}}>0$.
\end{proposition}

\begin{proof}
Let $J\subseteq\{1,\ldots,k-1\}$ be the set of ordinary coordinates at
which the residue at $1-u_j$ is taken.  Every ordinary coordinate outside
$J$ remains on
$\Re(a_j+u_j)=1-\eps_k$.  The signed coordinate has one of three forms:
it is the residue at $1/2-u_k$, a zero residue $\rho-u_k$, or it remains
on the final line
$\Re(a_k+u_k)=1/2-\eps_k$.  The three possibilities are summarized below.  In the zero row we count the
zero ordinate $\gamma$ as one vertical variable.
\[
\begin{array}{c|c|c}
\text{signed coordinate}&\text{vertical variables}&\Re s\\ \hline
1/2-u_k&k-1-|J|&
\alpha_k-\dfrac{(k-1-|J|)\eps_k}{k}-\dfrac{\sum u_i}{k}\\[5pt]
\rho-u_k&k-|J|&
\alpha_k-\dfrac{(k-1-|J|)\eps_k}{k}-\dfrac{\sum u_i}{k}\\[5pt]
\text{final line}&k-|J|&
\alpha_k-\dfrac{(k-|J|)\eps_k}{k}-\dfrac{\sum u_i}{k}
\end{array}
\]
For completeness, we now read the three rows separately.  If the signed
coordinate is the pole $1/2-u_k$, then
\[
 d=k-1-|J|,\qquad
 \Re s=\alpha_k-\frac{(k-1-|J|)\eps_k}{k}
 -\frac{u_1+\cdots+u_k}{k}.
\]
If it is a zero residue $\rho-u_k$, then
\[
 d=k-|J|,\qquad
 \Re s=\alpha_k-\frac{(k-1-|J|)\eps_k}{k}
 -\frac{u_1+\cdots+u_k}{k}.
\]
Finally, if the signed coordinate stays on its final line, then
\[
 d=k-|J|,\qquad
 \Re s=\alpha_k-\frac{(k-|J|)\eps_k}{k}
 -\frac{u_1+\cdots+u_k}{k}.
\]
In the zero-residue case, the extra dimension is the zero ordinate $\gamma$;
in the other cases all dimensions are continuous contour variables.

These formulas prove, in particular, the bound
\begin{equation}\label{eq:stratum-exponent}
 \Re s\leq
 \begin{cases}
 \displaystyle
 \alpha_k-\frac{(k-1-|J|)\eps_k}{k},
   &\text{signed residue},\\[6pt]
 \displaystyle
 \alpha_k-\frac{(k-|J|)\eps_k}{k},
   &\text{signed final line}.
 \end{cases}
\end{equation}
The two principal cases are therefore exactly the ones in which all
$k-1$ ordinary coordinates are replaced by their residues at $1$ and the
signed coordinate is replaced either by the pole at $1/2$ or by a zero of
$L(s,\chi)$.  Every other case loses at least
$\eps_k/k=1/(16k^2)$ in the exponent of $X$.  Thus the unwanted terms
have a \emph{fixed power saving}: they are smaller by a factor $X^{-c_k}$ for
some $c_k>0$, not merely by a logarithm.  The negative
$u$-term only makes the saving larger.

It remains to bound the vertical part of each stratum.  The total vertical
dimension $d$ in the table never exceeds $k$.  Proposition~\ref{prop:stratum-majorant}
selects $d$ independent kernel denominators and gives, uniformly in
$\boldsymbol u\in[0,\eta_k]^k$,
\[
 \text{stratum}\ll_k
 X^{\Re s}(\log T)^{M_k^{\mathrm{str}}},
\]
where in the zero-residue row the sum over all zeros, with multiplicity,
is already included.  There are only $3\cdot2^{k-1}$ structural types,
and the zero family is handled by the single zero-stratum estimate.
Consequently the total nonprincipal contribution is
\[
 \ll_k X^{\alpha_k-\eps_k/k}(\log T)^{B_k'}
\]
for some $B_k'>0$.  Since $\log T=(\log X)^2$, this proves
\eqref{eq:remainder-power}; for example one may take
$\kappa_k=1/(32k^2)$ after enlarging the logarithmic exponent.
\end{proof}

The contour shift gives the following decomposition: the smoothed count is
the deterministic residue at $1/2$, plus the residues coming from the zeros
of $L(s,\chi)$, plus an error that is smaller by a fixed power of $X$.

\begin{proposition}[Main contour decomposition]\label{prop:master-contour}
Let $X=N+\tfrac12$ with $N\in\N$ and
$T=\exp\{(\log X)^2\}$.  Recall the regularized sum
$\mathcal G_{k,T}$ from \eqref{eq:regularized-G}.
For $z=1/2$ put $b_z=1/2$, and for a zero
$z=\rho=1/2+i\gamma$ of $L(s,\chi)$ put $b_z=m_\rho$.  Set
\[
 \boldsymbol a_z(\boldsymbol u)
 :=(1-u_1,\ldots,1-u_{k-1},z-u_k)
\]
and
\begin{equation}\label{eq:main-residue-integral}
 \mathcal I_z(X,T):=
 -b_z\int_{[0,\eta_k]^k}
 X^{(k-1+z-u_1-\cdots-u_k)/k}
 e^{Q_k(\boldsymbol a_z(\boldsymbol u))/T^2}
 K_k(\boldsymbol a_z(\boldsymbol u))
 \dd\boldsymbol u.
\end{equation}
Then
\begin{equation}\label{eq:master-contour}
 \int_{[0,\eta_k]^k}
 \mathcal G_{k,T}(X;\boldsymbol u)\dd\boldsymbol u
 =
 \mathcal I_{1/2}(X,T)+\sum_\rho\mathcal I_\rho(X,T)
 +\mathcal R_k(X),
\end{equation}
where the zero sum and all contour integrals converge absolutely and
\begin{equation}\label{eq:master-remainder}
 \mathcal R_k(X)
 \ll_k
 X^{\alpha_k-\kappa_k}(\log X)^{M_k^{\mathrm{rem}}}.
\end{equation}
\end{proposition}

\begin{proof}
For a fixed $\boldsymbol u\in[0,\eta_k]^k$, apply the contour shifts
justified by Lemma~\ref{lem:shift-justification}.  Taking every ordinary
residue and the signed residue at $z$ gives the integrand in
\eqref{eq:main-residue-integral}; the residue signs and multiplicities are
those in Lemma~\ref{lem:prime-series}.  All remaining strata have the
uniform bound of Proposition~\ref{prop:remainders}.  Integration over the
fixed cube therefore gives \eqref{eq:master-remainder}.

Several limiting operations are present: the contour height tends to
infinity, the zero sum becomes infinite, and we integrate over
$\boldsymbol u$.  We take them in the following order, always under an absolute
majorant.

First, truncate every vertical contour at the same admissible height
$\tau_\nu$.  At this stage there are only finitely many contour pieces and
finitely many zero residues, so Cauchy's theorem gives a finite identity.
Second, apply Proposition~\ref{prop:stratum-majorant} to each term.  Its majorant is
independent of $\boldsymbol u$ and of $\tau_\nu$.  Third, integrate this
finite identity over $[0,\eta_k]^k$.  Fourth, let $\tau_\nu\to\infty$ in the
continuous contour terms.  Dominated convergence applies because of
\eqref{eq:uniform-stratum-integral}.  Fifth, let the truncated zero sum tend
to the full zero sum; Proposition~\ref{prop:stratum-majorant} gives an absolutely
summable majorant, again uniform in $\boldsymbol u$.  Finally, Tonelli's
theorem for the absolute majorant allows the $\boldsymbol u$-integral to be
interchanged with the full zero sum.  Hence every passage is made under an
absolute bound, and \eqref{eq:master-contour} follows with the stated
absolute convergence.
\end{proof}

To control the remaining error terms we need cancellation among primes in
the two residue classes modulo $4$.  RH for $L(s,\chi)$ gives the standard
estimate
\begin{equation}\label{eq:A-grh}
 A(y):=\sum_{p\le y}\chi(p)\ll y^{1/2}\log(2y).
\end{equation}
Partial summation gives, uniformly for $0\le\vartheta\le k\eta_k$ and for
any subinterval of $[Y,2Y]$,
\begin{equation}\label{eq:weighted-prime}
 \sum_{p\in I}\frac{\chi(p)}{p^\vartheta}
 \ll_k Y^{1/2-\vartheta}\log(2Y).
\end{equation}

\begin{lemma}[A simple exponent bound]\label{lem:exponent-optimization}
Let $0\le v_1\le\cdots\le v_r$ and let $e_1,\ldots,e_r$ be positive
integers.  If
\[
 \sum_{i=1}^r e_iv_i\le1,
\]
then
\begin{equation}\label{eq:exponent-optimization}
 \sum_{i<r}v_i+\frac12v_r\le1-\frac1{2r}.
\end{equation}
\end{lemma}

\begin{proof}
Since every $e_i\ge1$, we have $\sum_i v_i\le1$.  Since $v_r$ is the
largest of the $r$ numbers,
$v_r\ge r^{-1}\sum_i v_i$.  Therefore
\[
 \sum_{i<r}v_i+\frac12v_r
 =\sum_i v_i-\frac12v_r
 \le\left(1-\frac1{2r}\right)\sum_i v_i
 \le1-\frac1{2r}.
\]
\end{proof}

\begin{lemma}[Ordered dyadic bound]\label{lem:dyadic-bound}
Let $e_1+\cdots+e_r=k$, with all $e_i\ge1$, and let
$0\le\vartheta_i\le k\eta_k$.  Under RH for $L(s,\chi)$,
\begin{equation}\label{eq:dyadic-bound}
 \sum_{\substack{q_1<\cdots<q_r\\q_1^{e_1}\cdots q_r^{e_r}\le X}}
 \chi(q_r)\prod_{i=1}^r q_i^{-\vartheta_i}
 \ll_k X^{1-\frac1{2r}}(\log X)^{O_k(1)}.
\end{equation}
\end{lemma}

\begin{proof}
We group each $q_i$ into a dyadic interval $[Q_i,2Q_i)$.  This reduces the
range of prime sizes to finitely many boxes whose exponents can be optimized.
For fixed $q_1,\ldots,q_{r-1}$, the last prime lies in a subinterval of
$[Q_r,2Q_r)$, so \eqref{eq:weighted-prime} applies.  A box contributes
\[
 \ll_k\left(\prod_{i<r}Q_i^{1-\vartheta_i}\right)
 Q_r^{1/2-\vartheta_r}(\log X)^{O_k(1)}.
\]
Write $Q_i=X^{v_i}$.  The product condition gives
$\sum e_iv_i\le1$.  If we ignore the nonnegative $\vartheta_i$, the exponent
of $X$ in the box is
$\sum_{i<r}v_i+v_r/2$.  Lemma~\ref{lem:exponent-optimization} bounds this by
$1-1/(2r)$.  There are only $O_k((\log X)^r)$ boxes, proving the claim.
\end{proof}

We next collect three kinds of terms that are not part of the main residue
calculation.  The ``tail'' comes from large auxiliary variables
$u_j$; the ``diagonal'' comes from equal neighboring primes in the smoothed
ordered sum; and the repeated-prime term converts the squarefree calculation
back to the full fixed-$k$ layer.  All three will be smaller than the main
scale.

\begin{proposition}[Tail, diagonal, and repeated-prime errors]
\label{lem:large-u}\label{lem:repeated}\label{prop:diagonal}
Let $\eta=\eta_k$ and $\alpha_k=1-1/(2k)$.  Define
\[
 \mathcal T_k^\sharp(X;\eta):=
 \int_{[0,\infty)^k\setminus[0,\eta]^k}G_k(X;\boldsymbol u)\,\dd\boldsymbol u,
\]
The superscript $\sharp$ indicates that this is the sharp, unsmoothed tail.
\[
 \mathcal T_{k,T}(X;\eta):=
 \int_{[0,\infty)^k\setminus[0,\eta]^k}
 \mathcal G_{k,T}(X;\boldsymbol u)\,\dd\boldsymbol u.
\]
Then for some $\lambda_k>0$,
\begin{equation}\label{eq:sharp-tail-bound}
 \mathcal T_k^\sharp(X;\eta)
 \ll_k X^{\alpha_k-\lambda_k}(\log X)^{O_k(1)}.
\end{equation}
Moreover,
\begin{equation}\label{eq:regularized-tail-identity}
 \mathcal T_{k,T}(X;\eta)
 =\mathcal T_k^\sharp(X;\eta)+\mathcal T_k^\Delta(X;T,\eta)
 +O_{k,M}(X^{-M}),
\end{equation}
where $\mathcal T_k^\Delta(X;T,\eta)$ denotes the adjacent-equality part
of the weighted regularized sum, namely the same sum as
$\mathcal E_k^\Delta(X;T)$ with the extra factor $W_\eta(\boldsymbol p)$.
Under RH for $L(s,\chi)$,
\begin{equation}\label{eq:tail-diagonal-bound}
 \mathcal T_k^\Delta(X;T,\eta),\quad \mathcal E_k^\Delta(X;T)
 \ll_k X^{1-\frac1{2(k-1)}}(\log X)^{O_k(1)}
 \qquad(k\ge2).
\end{equation}
Finally, if
\[
 R_k(X):=\sum_{\substack{n\le X\\\Om(n)=k,\ \om(n)<k}}
 \chi(\Pplus(n)),
\]
then
\begin{equation}\label{eq:repeated}
 R_k(X)\ll_k X^{1-\frac1{2(k-1)}}(\log X)^{O_k(1)}\quad(k\ge3),
 \qquad R_2(X)=A(\sqrt X).
\end{equation}
\end{proposition}

\begin{proof}
Integrating over the complement of $[0,\eta]^k$ inserts the exact weight
\[
 W_\eta(\boldsymbol p)=1-\prod_{j=1}^k(1-p_j^{-\eta})
 =\sum_{\emptyset\ne S\subseteq\{1,\ldots,k\}}
 (-1)^{|S|+1}\prod_{j\in S}p_j^{-\eta}.
\]
Consider one nonempty $S$ and one dyadic box $P_i\le p_i<2P_i$.  By
\eqref{eq:weighted-prime}, its contribution is
\[
 \ll_k(P_1\cdots P_{k-1})P_k^{1/2}
 \prod_{j\in S}P_j^{-\eta}(\log X)^{O_k(1)}.
\]
Writing $P_i=X^{v_i}$ gives $0\le v_1\le\cdots\le v_k$ and
$\sum v_i\le1$.  Without the $S$-penalty the exponent is at most
$\alpha_k$.  If some penalized $v_j\ge1/(2k)$, we save at least
$\eta/(2k)$.  Otherwise either $S=\{k\}$ and $\sum v_i<1/2$, or a penalized
index $j<k$ has $v_j<1/(2k)$; then ordering and $\sum v_i=1$ at a maximum
force
\[
 v_k>\frac1k+\frac1{2k(k-1)},
\]
which saves at least $1/(4k(k-1))$ through the term $-v_k/2$.  Thus
\eqref{eq:sharp-tail-bound} holds, for example with
$\lambda_k=\min\{\eta/(2k),1/(4k(k-1))\}$.

The weighted desmoothing statement in Lemma~\ref{lem:desmoothing}, applied
to the finite expansion of $W_\eta$, gives
\eqref{eq:regularized-tail-identity}.  It remains only to identify the
adjacent-equality terms.  Partition such a prime tuple into its maximal
constant runs.  If the run lengths are $e_1,\ldots,e_r$, then
$e_1+\cdots+e_r=k$ with $r\le k-1$, and the corresponding block primes are
strictly ordered after desmoothing.  Every order kernel inside a constant run
has value $h_T(1)=1/2$, so a fixed block pattern contributes only the harmless
factor $2^{-(k-r)}$.  Moreover a monomial
$\prod_{j\in S}p_j^{-\eta}$ from $W_\eta$ collapses to
$\prod_{i=1}^r q_i^{-\vartheta_i}$ with
\[
 0\leq\vartheta_i=\eta\,\#(S\cap\text{the $i$th run})\leq k\eta.
\]
Thus every diagonal block is exactly of the form controlled by
Lemma~\ref{lem:dyadic-bound}, up to the same super-polynomial desmoothing
error.  This proves \eqref{eq:tail-diagonal-bound}.  The same block argument
with no smoothing gives \eqref{eq:repeated}; for $k=2$ the only repeated terms
are $p^2$, so the contribution is $A(\sqrt X)$.
\end{proof}

\section{The fixed-layer explicit formula}\label{sec:main-proof}

The contour argument reduces the problem to the principal residue integrals
$\mathcal I_z(X,T)$ in Proposition~\ref{prop:master-contour}.  After the
logarithmic weights are introduced, the $u$-integral is exponentially
concentrated near $\boldsymbol u=\boldsymbol0$, so its leading term comes from
the remaining smooth factor at that endpoint.

\begin{lemma}[Main residue evaluation]\label{lem:endpoint}
For $z=1/2$, the residue integral is
\begin{equation}\label{eq:endpoint-half}
 -\frac12 C_k(1/2)
 \frac{X^{\alpha_k}}{(\log X)^k}
 +O_k\left(\frac{X^{\alpha_k}}{(\log X)^{k+1}}\right).
\end{equation}
For the zero residues, counted with multiplicity, one has
\begin{equation}\label{eq:endpoint-zeros}
 \sum_\rho \mathcal I_\rho(X,T)
 =
 -\frac{X^{\alpha_k}}{(\log X)^k}
 \sum_\rho m_\rho C_k(\rho)X^{i\gamma/k}
 +O_k\left(\frac{X^{\alpha_k}}{(\log X)^{k+1}}\right).
\end{equation}
The zero series in \eqref{eq:endpoint-zeros} converges absolutely.
\end{lemma}

\begin{proof}
Write $L=\log X$ and
\[
 \Phi_{z,T}(\boldsymbol u)
 =e^{Q_k(\boldsymbol a_z(\boldsymbol u))/T^2}
 K_k(\boldsymbol a_z(\boldsymbol u)).
\]
After factoring $X^{(k-1+z)/k}$ out of
\eqref{eq:main-residue-integral}, the remaining integral is
\[
 J_{z,T}(L):=
 \int_{[0,\eta_k]^k}
 e^{-(L/k)(u_1+\cdots+u_k)}
 \Phi_{z,T}(\boldsymbol u)\dd\boldsymbol u.
\]

For a zero residue $z=1/2+i\gamma$,
Lemma~\ref{lem:principal-zero-decay} gives the $k$ denominator factors
needed for the zero decay.  For the deterministic pole $z=1/2$, the same
bound with $\gamma=0$ follows directly from Lemma~\ref{lem:geometry}.
Together with the Gaussian factor, uniformly on the $\boldsymbol u$-cube,
\begin{equation}\label{eq:Phi-derivatives}
 |\Phi_{z,T}(\boldsymbol u)|
 +\|\nabla\Phi_{z,T}(\boldsymbol u)\|
 \ll_k
 (1+|\gamma|)^{-k}e^{-c_k\gamma^2/T^2}.
\end{equation}
Indeed,
\[
 \Re Q_k(\boldsymbol\sigma+i\boldsymbol t)
 =Q_k(\boldsymbol\sigma)-Q_k(\boldsymbol t),
\]
while on the line in question $Q_k(\boldsymbol t)=c_k'\gamma^2$.
Differentiating the Gaussian produces one extra factor of size
$O_k((1+|\gamma|)T^{-2})$.  This does not spoil the bound.  Indeed, for
any $c'>0$ smaller than the Gaussian constant $c$, the elementary inequality
\[
 y e^{-cy^2}\ll_{c,c'} e^{-c'y^2}\qquad(y\ge0)
\]
absorbs the extra factor after writing $y=|\gamma|/T$.  Thus
\eqref{eq:Phi-derivatives} holds after decreasing the constant $c_k$ if
necessary.  We use the same symbol $c_k$ for the smaller positive constant.

Because of the factor
$e^{-(L/k)(u_1+\cdots+u_k)}$, values of $u_j$ much larger than $1/L$ are
exponentially suppressed.  Thus only a small neighborhood of
$\boldsymbol u=\boldsymbol0$ matters at leading order.  Make the change of
variables $v_j=Lu_j/k$.  The mean-value theorem,
\eqref{eq:Phi-derivatives}, and the exponentially small tail outside
$[0,\eta_kL/k]^k$ give
\begin{equation}\label{eq:laplace-endpoint}
 J_{z,T}(L)
 =
 \left(\frac{k}{L}\right)^k\Phi_{z,T}(\boldsymbol0)
 +O_k\left(
 \frac{(1+|\gamma|)^{-k}e^{-c_k\gamma^2/T^2}}
 {L^{k+1}}\right).
\end{equation}
The mean-value error is bounded by the derivative estimate in
\eqref{eq:Phi-derivatives}, while the omitted tail is exponentially small in
$L$.

By Lemma~\ref{lem:kernel-value},
\[
 k^kK_k(1,\ldots,1,z)=C_k(z).
\]
Equation \eqref{eq:laplace-endpoint} therefore proves
\eqref{eq:endpoint-half}, because at $z=1/2$ the Gaussian differs from
one by $O_k(T^{-2})$.

For the zeros, the zero-counting estimate gives
\[
 \sum_\rho m_\rho(1+|\gamma|)^{-k}<\infty
 \qquad(k\geq2),
\]
so the errors in \eqref{eq:laplace-endpoint} may be summed absolutely.
The contour regularization leaves an auxiliary Gaussian factor in each zero
coefficient.  We now remove this factor.  The worst decay occurs for $k=2$, and therefore
it is enough to use
\[
 |C_k(1/2+i\gamma)|\ll_k(1+|\gamma|)^{-2}.
\]
For $|\gamma|\le T$,
\[
 \left|e^{Q_k(1,\ldots,1,\rho)/T^2}-1\right|
 \ll_k\frac{1+\gamma^2}{T^2}.
\]
Hence, using
$N_\chi(t+1)-N_\chi(t)\ll\log(2+t)$,
\[
 \sum_{|\gamma|\le T}m_\rho |C_k(\rho)|
 \left|e^{Q_k(1,\ldots,1,\rho)/T^2}-1\right|
 \ll_k \frac{1}{T^2}\sum_{n\le T}\log(2+n)
 \ll_k \frac{\log(2T)}{T}.
\]
For $|\gamma|>T$, the Gaussian factor is bounded in absolute value by a
constant depending only on $k$ (in fact by $e^{O_k(T^{-2})}$ times the
decaying Gaussian), so
\[
 \left|e^{Q_k(1,\ldots,1,\rho)/T^2}-1\right|\ll_k1.
\]
Therefore
\[
 \sum_{|\gamma|>T}m_\rho |C_k(\rho)|
 \left|e^{Q_k(1,\ldots,1,\rho)/T^2}-1\right|
 \ll_k\sum_{n>T}\frac{\log(2+n)}{n^2}
 \ll_k\frac{\log(2T)}{T}.
\]
Thus the whole Gaussian-removal error is
$O_k(\log(2T)/T)$, which is stronger than the bound needed here.  Since
$T=\exp\{L^2\}$, it is in particular $O_k(L^{-1})$.  Substitution in
\eqref{eq:laplace-endpoint} proves \eqref{eq:endpoint-zeros} and absolute
convergence.
\end{proof}

\begin{proof}[Proof of Theorem~\ref{thm:main-intro}]
Let $N=\lfloor x\rfloor$, take $X=N+\tfrac12$, and put
$T=\exp\{(\log X)^2\}$.  Set
\[
 \mathcal J_{\rm full}:=
 \int_{[0,\infty)^k}\mathcal G_{k,T}(X;\boldsymbol u)
 \dd\boldsymbol u,
 \qquad
 \mathcal J_{\rm cube}:=
 \int_{[0,\eta_k]^k}\mathcal G_{k,T}(X;\boldsymbol u)
 \dd\boldsymbol u.
\]
By definition,
$\mathcal J_{\rm cube}=\mathcal J_{\rm full}
-\mathcal T_{k,T}(X;\eta_k)$.  Lemma~\ref{lem:desmoothing} gives
\[
 \mathcal J_{\rm full}
 =D_k^\square(X)+\mathcal E_k^\Delta(X;T)
  +O_{k,M}(X^{-M}),
\]
while Lemma~\ref{lem:large-u} gives the signed identity
\[
 \mathcal T_{k,T}(X;\eta_k)
 =\mathcal T_k^\sharp(X;\eta_k)
  +\mathcal T_k^\Delta(X;T,\eta_k)
  +O_{k,M}(X^{-M}).
\]
Finally, Proposition~\ref{prop:master-contour} evaluates
$\mathcal J_{\rm cube}$.  Combining these three exact relations gives a
single bookkeeping identity, written explicitly so that each error term is
visible:
\begin{align}
 D_k^\square(X)
 ={}&\mathcal I_{1/2}(X,T)+\sum_\rho\mathcal I_\rho(X,T)
 +\mathcal R_k(X)
 +\mathcal T_k^\sharp(X;\eta_k)
 +\mathcal T_k^\Delta(X;T,\eta_k)\notag\\
 &\hspace{24mm}-\mathcal E_k^\Delta(X;T)
 +O_{k,M}(X^{-M}).
 \label{eq:error-ledger-squarefree}
\end{align}
The terms in this ledger satisfy
\begin{align}
 \mathcal R_k(X)
 &\ll_k
 X^{\alpha_k-\kappa_k}(\log X)^{M_k^{\mathrm{rem}}},
 \label{eq:error-cont}\\
 \mathcal T_k^\sharp(X;\eta_k)
 &\ll_k
 X^{\alpha_k-\lambda_k}(\log X)^{M_k^{\mathrm{tail}}},
 \label{eq:error-tail}\\
 \mathcal T_k^\Delta(X;T,\eta_k),\quad
 \mathcal E_k^\Delta(X;T)
 &\ll_k
 X^{1-\frac1{2(k-1)}}(\log X)^{M_k^\Delta}.
 \label{eq:error-diag}
\end{align}
Here and below the finitely many logarithmic exponents may be enlarged
without further comment.

So far the contour calculation has counted only integers with $k$ distinct
prime factors.  To return to the full condition $\Om(n)=k$, we add back the
integers in which some prime repeats.  This is the exact identity
\begin{equation}\label{eq:error-repeated-ledger}
 D_k(X)=D_k^\square(X)+R_k(X),
\end{equation}
where $R_k$ is defined in Lemma~\ref{lem:repeated}.  For $k\geq3$ it
satisfies \eqref{eq:repeated}, and for $k=2$,
$R_2(X)=A(\sqrt X)\ll X^{1/4}\log(2X)$.  The diagonal exponent has the
strict gap
\[
 \alpha_k-\left(1-\frac1{2(k-1)}\right)
 =\frac1{2k(k-1)}>0\qquad(k\geq3),
\]
and for $k=2$ the diagonal exponent $1/2$ and repeated-prime exponent
$1/4$ are both below $\alpha_2=3/4$.  It follows from
\eqref{eq:error-ledger-squarefree}--\eqref{eq:error-repeated-ledger} that
there is $\delta_k>0$ such that the total of all nonprincipal, tail,
diagonal, repeated-prime, and desmoothing errors is
\begin{equation}\label{eq:power-errors-absorbed}
 O_k\left(X^{\alpha_k-\delta_k}(\log X)^{M_k^\ast}\right)
 =o_k\left(\frac{X^{\alpha_k}}{(\log X)^{k+1}}\right).
\end{equation}

Lemma~\ref{lem:endpoint} shows that the pole $z=1/2$ contributes
\[
 -\frac12C_k(1/2)
 \frac{X^{\alpha_k}}{(\log X)^k}
 +O_k\left(\frac{X^{\alpha_k}}{(\log X)^{k+1}}\right).
\]
The identity
\[
 \frac12C_k(1/2)=
 \frac{2^{k-1}k^{2k-1}}{(k-1)!(2k-1)}=B_k
\]
identifies the deterministic term.  The zeros contribute
\[
 -\frac{X^{\alpha_k}}{(\log X)^k}
 \sum_{\rho}m_\rho C_k(\rho)X^{i\gamma/k}
+O_k\left(\frac{X^{\alpha_k}}{(\log X)^{k+1}}\right).
\]
Pairing conjugate zeros gives the series in \eqref{eq:main-intro}.
Moreover,
$\sum_\rho m_\rho|C_k(\rho)|<\infty$ and
$|e^{i\gamma y/k}|=1$ for real $y$; the Weierstrass $M$-test therefore
proves absolute and uniform convergence as a function of $y=\log X$.
Together with \eqref{eq:error-ledger-squarefree} and
\eqref{eq:power-errors-absorbed}, this proves the asserted formula with
$X$ in place of $x$.

Finally, $D_k(x)=D_k(X)$.  The deterministic factor changes by an admissible
$O_k(x^{\alpha_k-1}(\log x)^{-k})$.  For the zero series put
$h=k^{-1}\log(X/x)=O_k(x^{-1})$.  Zero counting and
$C_k(1/2+i\gamma)\ll_k(1+|\gamma|)^{-k}$ give
\[
 \sum_{\rho}m_\rho|C_k(\rho)|
 |e^{i\gamma h}-1|
 \ll_k
 \begin{cases}
 |h|\log^2(2/|h|),&k=2,\\
 |h|,&k\geq3.
 \end{cases}
\]
To see this, use $|e^{it}-1|\leq\min\{2,|t|\}$ and split the
zero sum at $|\gamma|=|h|^{-1}$.  After multiplication by
$x^{\alpha_k}(\log x)^{-k}$, the resulting change is
$O_k(x^{\alpha_k-1})$ for $k=2$ and smaller for $k\geq3$.
It is therefore absorbed by the stated error term.
\end{proof}

\section{Why the deterministic term wins}\label{sec:sign}

The explicit formula contains infinitely many oscillating terms, one from
each zero of $L(s,\chi)$.  Predicting how all of their phases line up would be
far harder than we need.  Instead we use a worst-case bound: replace every
oscillation by its absolute value and add them all.  If even this exaggerated
total is smaller than the fixed deterministic term, then the sign can never
be reversed once the lower-order error is small enough.

The point of the next two lemmas is that no information about the phases of
the zeros is required.  We first add the zero weights
$(1/4+\gamma^2)^{-1}$ exactly, and then compare each coefficient $C_k(\rho)$
with that weight.  This turns an infinite oscillating problem into one
numerical inequality.

Absolute convergence follows immediately from
\begin{equation}\label{eq:C-decay}
 C_k(1/2+i\gamma)\ll_k(1+|\gamma|)^{-k}
\end{equation}
and $N_\chi(T)\ll T\log(2T)$.  We now prove the stronger estimate required
for Corollary~\ref{cor:negative-intro}.

We collect the zeros in one positive quantity, avoiding separate estimates
for the low zeros.  Let
\[
 \Lambda(s):=
 \left(\frac4\pi\right)^{(s+1)/2}
 \Gamma\left(\frac{s+1}{2}\right)\beta(s).
\]
\begin{lemma}[Total zero mass]\label{lem:zero-mass}
Under RH for $L(s,\chi)$,
\begin{equation}\label{eq:zero-mass}
 \mathcal S_\chi:=
 \sum_{\gamma>0}\frac{m_\rho}{1/4+\gamma^2}
 =\frac{\Lambda'}{\Lambda}(1)
 =\frac12\log\frac4\pi-\frac{\gamma_0}{2}
 +\frac{\beta'(1)}{\beta(1)},
\end{equation}
where $\gamma_0$ is Euler's constant.  The alternating series for
$\beta'(1)$ has decreasing terms after its first nonzero term, so
\[
 0<\beta'(1)<\frac{\log3}{3},\qquad \beta(1)=\frac{\pi}{4}.
\]
These bounds imply directly
\begin{equation}\label{eq:S-bound}
 \mathcal S_\chi<\frac13.
\end{equation}
\end{lemma}

\begin{proof}
The completed function $\Lambda$ is entire of order one and satisfies the
symmetry $\Lambda(s)=\Lambda(1-s)$.  Since $\beta(1/2)>0$, there is no
zero at the center $s=1/2$.  Put $\Xi(z)=\Lambda(1/2+z)$.  This function is
even.  Its Hadamard product---the standard factorization of an entire
function in terms of its zeros---allows conjugate zeros to be paired, which
is legitimate because
$\sum_\gamma(1+\gamma^2)^{-1}<\infty$, gives
\[
 \frac{\Xi'}{\Xi}(z)
 =\sum_{\gamma>0}\frac{2zm_\rho}{z^2+\gamma^2}.
\]
Taking $z=1/2$ proves the first equality in \eqref{eq:zero-mass}.
Logarithmic differentiation of the displayed definition of $\Lambda$,
using $\psi=\Gamma'/\Gamma$ and $\psi(1)=-\gamma_0$, proves the second.

Since $t\mapsto(\log t)/t$ decreases for $t\geq3$,
\[
 \beta'(1)=
 \frac{\log3}{3}-\frac{\log5}{5}
 +\frac{\log7}{7}-\cdots
\]
is an alternating series with decreasing terms.  Hence
$0<\beta'(1)<(\log3)/3$, while $\beta(1)=\pi/4$.
Therefore
\[
 \mathcal S_\chi
 <\frac12\log\frac4\pi-\frac{\gamma_0}{2}
   +\frac{4\log3}{3\pi}
 <\frac13.
\]
For example, $\pi>3.14$, $\gamma_0>0.57$, $\log3<1.10$, and
$\log(4/3.14)<0.25$ give the last inequality directly.  This proves
\eqref{eq:S-bound}.
\end{proof}

\begin{lemma}[Coefficient comparison]\label{lem:C-ratio}
For $k\geq2$ and $\rho=1/2+i\gamma$,
\[
 \frac{|C_k(\rho)|}{C_k(1/2)}
 \leq\frac{3/4}{1/4+\gamma^2}.
\]
\end{lemma}

\begin{proof}
Put $t=4\gamma^2$.  Directly from \eqref{eq:B-C-def},
\[
 \frac{|C_k(1/2+i\gamma)|}{C_k(1/2)}
 =
 \left(1+\frac{t}{(2k-1)^2}\right)^{-1/2}
 (1+t)^{-(k-1)/2}.
\]
After multiplication by $(1+t)/3$, the right-hand side is at most $1$.
For $k\geq3$ this is immediate because
$(1+t)^{(3-k)/2}\leq1$.  For $k=2$ it follows after squaring from
$(1+t)/(1+t/9)\leq9$.
\end{proof}

\begin{proof}[Proof of Corollary~\ref{cor:negative-intro}]
Since $C_k(1/2)=2B_k$, Lemma~\ref{lem:C-ratio} and
\eqref{eq:S-bound} give
\[
 2\sum_{\gamma>0}m_\rho|C_k(\rho)|
 \leq\frac32C_k(1/2)\mathcal S_\chi
 =3B_k\mathcal S_\chi<B_k.
\]
Consequently, even if every zero term points in the direction most favorable
to cancelling the deterministic term, the expression in braces in
\eqref{eq:main-intro} is still at least the fixed positive number
$B_k-2\sum_{\gamma>0}m_\rho|C_k(\rho)|$.  The outside minus sign 
forces a negative main term.  Since the remaining error is smaller by a
factor of $\log x$, it cannot change the sign once $x$ is sufficiently
large.
\end{proof}

\end{document}